\documentclass[12pt,reqno]{amsart}

\usepackage[colorlinks=true,linkcolor=blue,citecolor=blue]{hyperref}
\usepackage{array}
\usepackage{amssymb,latexsym}
\usepackage{bm,bbding}
\usepackage{booktabs}
\usepackage{cleveref}
\usepackage{graphicx,subcaption}
\usepackage{pstricks-add}
\usepackage[all]{xy}

\makeatletter
\@namedef{subjclassname@2020}{%
  \textup{2020} Mathematics Subject Classification}
\makeatother

\newtheorem{theorem}{Theorem}[section]
\newtheorem{lemma}[theorem]{Lemma}
\newtheorem{proposition}[theorem]{Proposition}
\newtheorem{corollary}[theorem]{Corollary}

\numberwithin{equation}{section}
\theoremstyle{definition}

\newtheorem{question}[theorem]{Question}

\newcommand{\cp}{\mathbb{C}P}
\newcommand{\zz}{\mathbb{Z}}

\newcommand{\sign}{{\rm sign}}

\begin{document}

\title[Circle actions on manifolds with three fixed points]{Non-existence of circle actions on oriented manifolds with three fixed points except in dimensions $4$, $8$ and $16$}{\thanks{This project was partially supported by the the National Natural Science Foundation of China (No. 12071337).}
\author{Hao Dong, Jianbo Wang}
\address{(Hao Dong) School of Mathematics, Tianjin University, Tianjin 300350, China}
\email{2021233011@tju.edu.cn}
\address{(Jianbo Wang) School of Mathematics, Tianjin University, Tianjin 300350, China}
\email{wjianbo@tju.edu.cn}

\begin{abstract}
Let $M$ be a  smooth $4k$-dimensional orientable closed  manifold, and assume that $M$ has at most two non-zero Pontrjagin numbers which are associated to the top dimensional Pontrjagin class and the square of the middle dimensional Pontrjagin class.
We prove that the signature of $M$ being equal to $1$ and the $\hat{A}$-genus of $M$ vanishing cannot hold at the same time except $\dim M=8, 16$.  As an application, we claim that the dimensions of oriented $S^1$-manifolds with exactly three fixed points are only $4, 8$ and $16$, and the rational projective plane whose dimension is greater than $16$ has no smooth non-trivial $S^1$-action.
\end{abstract}
\date{\today}

\maketitle


\section{Introduction}

Unless otherwise stated, all the manifolds mentioned in this paper are smooth connected, and all circle actions on the manifolds are smooth and effective.
 
Let $M$ be a $4k$-dimensional orientable closed manifold with only two possible non-zero Pontrjagin numbers $p_{\frac{k}{2}}^2[M]:=(p^{}_{\frac{k}{2}}\cdot p^{}_{\frac{k}{2}})[M]$ ($k$ even) and $p^{}_{k}[M]$ (for abbreviation, $p^{}_j: = p^{}_j(TM)$).  Note that the Pontrjagin numbers of an orientable manifold are integers. For example, the following manifolds have at most two non-zero Pontrjagin numbers:
\begin{itemize}
\item a class of highly connected manifolds, such as  $(2k-1)$-connected $4k$-dimensional manifolds;
\item a $4k$-dimensional $S^1$-manifold with exactly three isolated fixed points (\cite[Lemma 3.1]{Wiemeler23});
\item a $8k$-dimensional rational projective plane (see \Cref{sec-RPP} or \cite{Su14, FS, KenSu19}).
\end{itemize}
Such manifolds with only two possible non-zero Pontrjagin numbers have a feature that whose signature and $\hat{A}$-genus have simpler forms described as the rational linear combination of two Pontrjagin numbers (see \Cref{sec-signAhat}).

The key lemma of this paper is listed below.
\begin{lemma}\label{NOsign1Ahat0}
Let $M$ be a $4k$-dimensional oriented closed manifold with at most two non-zero Pontrjagin numbers $p_{\frac{k}{2}}^2[M]$ ($k$ even) and $p^{}_{k}[M]$.
When $k\geqslant 5$ or $k=1,3$, $\sign(M)=1$ and $\hat{A}(M)=0$ cannot hold at the same time. 
\end{lemma}

In the dimensions $8$  and $16$, the projective planes $\mathbb{H}P^2$ and $\mathbb{O}P^2$ are both spin manifolds, and they are the examples with the signature being $1$ and vanishing $\hat{A}$-genus.

As an application of the above lemma, we discuss the circle actions on the manifold $M^{4k}$ with possible two non-zero Pontrjagin numbers $p_{\frac{k}{2}}^2[M]$ ($k$ even) and $p^{}_{k}[M]$. Let $M$ be a  manifold admitting a non-trivial $S^1$-action, and $M^{S^1}$ be the fixed point set.  Li \cite{Li2015} asked the following question.
\begin{question}\cite[Question 1.3]{Li2015}\label{Q-Liques1.3}
Given any two positive integers $k$ and $n$, whether or not there exists an orientable $S^1$-manifold $M^{4n}$ such that $M^{S^1}$ consists of $2k + 1$ fixed points?
\end{question}

Since the fixed points are isolated, the dimension of $M$ must be even (\cite[Remark 1.2]{Li2015}). For any two positive integers $2k$ and $n$, there exists an orientable $S^1$-manifold $M^{2n}$ such that $M^{S^1}$ consists of $2k$ fixed points (\cite[Example 2.1 \& Proposition 2.4]{Li2015}). For the case that $M^{S^1}$ has $2k+1$ fixed points, then $k\geqslant 1$, and Li proved that the dimension of $M$ is divisible by $4$ because of the non-zero signature. So, as a realization problem of the number of isolated fixed points and the dimension of $S^1$-manifolds, \Cref{Q-Liques1.3} are the real question to be concerned with. If there exists an orientable $S^1$-manifold $M^{4n}$ such that $M^{S^1}$ consists of three fixed points, then for any positive integer $k\geqslant 2$, Li also proved that one can construct an orientable $S^1$-manifold $N^{4n}$ from $M^{4n}$ such that $N^{S^1}$ consists of $k$ fixed points (\cite[Theorem 1.4]{Li2015}). Now, the \Cref{Q-Liques1.3} is reduced from any $k$ to $k=1$. 
\begin{question}\label{Q-Noexis3fp}
Whether or not there exists an orientable $S^1$-manifold $M^{4n}$ such that $M^{S^1}$ consists of three fixed points?
\end{question}

More details about the \Cref{Q-Liques1.3} and \Cref{Q-Noexis3fp} may be referenced in \cite{Li2015, Jang2021, Wiemeler23}. In the following, suppose that the $S^1$-manifold $M$ has exactly three fixed points. The known examples of orientable manifolds admitting a non-trivial circle action with exactly three fixed points are the projective planes $\cp^2$, $\mathbb{H}P^2$ and $\mathbb{O}P^2$. 
\begin{itemize}
\item If $\dim M = 4$, the weights of fixed points (or fixed point datum, see \Cref{sec-S1mfd3fpts}) agree with a certain linear circle action on $\mathbb{C}P^2$ (\cite[Theorem 1.8]{Li2015}, \cite[Theorem 7.1]{Jang2018}, \cite[Theorem 2.4]{Jang2020} or \cite[Corollary 4.3]{Jang2021}). 
\item If $\dim M=8$, the weights of fixed points agree with a certain circle action on $\mathbb{H}P^2$ (\cite[Theorem 1]{Kustarev2014}, \cite[Theorem 1.6]{Jang2021}).
\item
Non-existence of $12$-dimensional oriented $S^1$-manifolds with exactly three fixed points has been shown by Jang \cite{Jang2021}.
\item
Wiemeler further narrowed down the possibility of a $S^1$-manifold with exactly three fixed points in \cite[Theorem 1.3 \& Lemma 3.1]{Wiemeler23}: 
if $M$ is an orientable $S^1$-manifold with exactly three fixed points, then $\dim M = 4\cdot 2^a$ or $\dim M = 8\cdot (2^a +2^b)$ with $a, b\in\zz_{\geqslant 0}$, $a\ne b$; Moreover, all Pontrjagin numbers of $M^{4k}$ except maybe $p^2_{\frac{k}{2}}[M]$ and $p^{}_k[M]$ vanish.
\end{itemize}
Based on the work of Jang \cite{Jang2021} and Wiemeler \cite{Wiemeler23}, we prove the following main theorem that if $M$ is an orientable closed $S^1$-manifold with exactly three fixed points and $\dim M>0$, then $\dim M=4,8,16$. In each dimension,  the projective space $\cp^2$, $\mathbb{H}P^2$ and $\mathbb{O}P^2$ are the examples, respectively.
\begin{theorem}\label{thm-Noexis3fixedpts}
Let $M$ be a $4k$-dimensional orientable closed manifold. When $k\geqslant 5$ or $k=3$,  
there does not exist a non-trivial circle action on $M$ with exactly three fixed points.
\end{theorem}
\Cref{thm-Noexis3fixedpts} presents a negative answer to the \Cref{Q-Noexis3fp} except the dimensions $4, 8$ and $16$. In \cite{Wiemeler23}, Wiemeler mentioned that there are some similarities between those dimensions which support a rational projective plane and those dimensions which support an orientable $S^1$-manifold with three fixed points. The statement above  is based on the following observation: a circle action on an orientable closed manifold $M$ with exactly
three fixed points is \emph{equivariantly formal}, i.e. 
\[
\sum_{i\geqslant 0}b_i(M)=\sum_{i\geqslant 0}b_i(M^{S^1}), 
\]
if and only if $M$ is a rational projective plane, where $b_i(M)$ denotes the $i$-th Betti number of $M$.  

In \Cref{sec-RPP}, we prove that the rational projective plane whose dimension is greater than 16 has no smooth non-trivial $S^1$-action. Note that, the projective plane $\cp^2$, $\mathbb{H}P^2$ and $\mathbb{O}P^2$ all have non-trivial smooth $S^1$-action with exactly three fixed points.

\begin{theorem}\label{thm-rppnoS1}
The dimensions of rational projective planes who have smooth non-trivial $S^1$-action are $4$, $8$ and $16$.
\end{theorem}

The paper is organized as follows: In \Cref{sec-signAhat}, we discuss the signature and $\hat{A}$-genus of a manifold with only two possible non-zero Pontrjagin numbers and prove \Cref{NOsign1Ahat0}.
The main \Cref{thm-Noexis3fixedpts} is proved in \Cref{sec-S1mfd3fpts}. In \Cref{sec-RPP}, we prove \Cref{thm-rppnoS1}.

\section{Manifolds with at most two non-zero Pontrjagin numbers}\label{sec-signAhat}

By the Hirzebruch signature theorem (\cite[Theorem 8.2.2]{H} or \cite[Theorem 19.4]{MilnorStash}), which expresses the signature of an orientable closed  manifold $M^{4k}$  as the $L$-genus, 
\[
\sign(M)=L_{k}[M],
\]
where 
\[
L=1+L_1(p^{}_1)+L_2(p^{}_1,p^{}_2)+L_3(p^{}_1,p^{}_2,p^{}_3)+\cdots
\]
 is the Hirzebruch $L$-polynomial in the Pontrjagin classes of $M^{4k}$, and   the multiplicative sequence  $\{L_i(p^{}_1, \dots, p^{}_i)\}$ of $L$-polynomial  belongs to the power series 
 \[
 Q^{}_L(z)=\dfrac{\sqrt{z}}{\tanh\sqrt{z}}.
\]
\begin{lemma}\label{H}\cite[Lemma 1.4.1]{H}
Let $\{K_j(p^{}_1,\dots,p^{}_j)\}$ be the multiplicative sequence corresponding to the power series $Q(z)=\sum_{i=0}^\infty b_iz^i$.
Then the coefficient of $p^{}_k$ in $K_k$ is equal to $s_k$, and the coefficient of $p_{\frac{k}{2}}^2$ ($k$ even) in $K_{k}$ is $\frac{1}{2}\left(s_{\frac{k}{2}}^2-s^{}_{k}\right)$, where the $s^{}_k$ can be calculated by a formula of Cauchy:
\[
1-z\frac{d}{dz}\log Q(z)=Q(z)\frac{d}{dz}\left(\frac{z}{Q(z)}\right)=\sum_{j=0}^\infty(-1)^js^{}_jz^j.
\]
\end{lemma}  

If $M$ is an orientable closed manifold with at most two non-zero Pontrjagin numbers,  by \Cref{H}, 
\begin{equation}\label{eq-signM}
\sign(M)=
\begin{cases}
s^{}_{k}p^{}_{k}[M]+s^{}_{\frac{k}{2},\frac{k}{2}}p_{\frac{k}{2}}^2[M], & k~\text{even};\\
s^{}_kp^{}_k[M], & k~\text{odd},
\end{cases}
\end{equation}
where 
\[
s_k=\dfrac{2^{2k}(2^{2k-1}-1)|B_{2k}|}{(2k)!}
\]
denotes the coefficient of $p^{}_k$  (see also \cite[Problem 19-C]{MilnorStash})  and 
\[
s^{}_{\frac{k}{2},\frac{k}{2}}=\frac{s_{\frac{k}{2}}^2-s^{}_{k}}{2}
\]
denotes the coefficient of $p^2_{\frac{k}{2}}$  (see also \cite[Lemma 1.5]{Anderson69}) in the $k^{\rm th}$ $L$-polynomial, respectively, and $B_{2k}$ denotes the $2k^{\mathrm{th}}$ Bernoulli numbers: $B_2=\frac{1}{6}$, $B_4=-\frac{1}{30}$, $B_6=\frac{1}{42},\dots$. 
The calculations of the Bernoulli numbers and the first several Bernoulli numbers can be found in \cite{Kellner04} and \cite{Wagstaff02}.

Similarly, Hirzebruch (\cite[p.197]{H}) defined the $\hat{A}$-sequence of $M^{4k}$ as a certain polynomials in the Pontrjagin classes of $M^{4k}$. More concretely, the power series
\[
Q^{}_{\hat{A}}(z)=\frac{\frac{1}{2}\sqrt{z}}{\sinh \left(\frac{1}{2}\sqrt{z}\right)}
\]
defines a multiplicative sequence $\left\{\hat{A}_{j}(p^{}_{1}, \ldots, p^{}_{j})\right\}$, and the Hirzebruch $\hat{A}$-polynomial is 
\[
\hat{A}=1+\hat{A}_1(p^{}_1)+\hat{A}_2(p^{}_1,p^{}_2)+\hat{A}_3(p^{}_1,p^{}_2,p^{}_3)+\cdots
\]
The $\hat{A}$-genus $\hat{A}(M)$ is the genus associated with the power series $Q^{}_{\hat{A}}(z)$ as above, and 
\[
\hat{A}(M)=\hat{A}_k(p^{}_1,\dots,p^{}_k)[M].
\]

In the case that the manifold has at most two non-zero Pontrjagin numbers, by Lemma \ref{H}, 
\begin{equation}\label{eq-Ahatgenus}
\hat{A}(M)=
\begin{cases}
a^{}_{k}p^{}_{k}[M]+a^{}_{\frac{k}{2},\frac{k}{2}}p_{\frac{k}{2}}^2[M], & k~\text{even};\\
a^{}_{k}p^{}_{k}[M], & k~\text{odd}.
\end{cases}
\end{equation}
where 
\[
a^{}_k=-\dfrac{|B_{2k}|}{2(2k)!}
\]
denotes the coefficient of $p^{}_k$  and 
\[
a^{}_{\frac{k}{2},\frac{k}{2}}=\frac{a_{\frac{k}{2}}^2-a^{}_{k}}{2}
\]
denotes the coefficient of $p^2_k$ in the $k^{\rm th}$ $\hat{A}$-polynomial, respectively. Note that the above content in this section can also be referenced in \cite{FS, Hu2023}.

In the proof of \Cref{NOsign1Ahat0} we need some properties of Bernoulli numbers, which will follow from the lemma below.

\begin{lemma}\label{On}\cite[Proposition 5.4]{Kellner04}
 Let $n$ be an even positive integer, then the Bernoulli number 
\[
B_n=(-1)^{\frac{n}{2}-1}\cdot\frac{\mathop{\prod}\limits_{p-1\nmid n}p^{\tau(p,n)+\mathrm{ord}_pn}}{\mathop{\prod}\limits_{p-1 | n} p},
\]
where $p$ is prime, $\mathrm{ord}_pn$ is the number of times $p$ divides $n$, and 
\[
\tau(p,n):=\sum_{\nu=1}^{\infty}\#(\Psi_{\nu}^{\mathrm{irr}}\cap\{(p, n\mod \varphi(p^\nu))\}),
\]
Here, ``$x \mod y$" gives the least non-negative residue $x^\prime$ with $0\leqslant x^\prime < y$ for positive integers $x$ and $y$,  and $\Psi_{\nu}^{\mathrm{irr}}$ denotes the set of irregular pairs of order $\nu$, $\varphi(p^\nu)$ is defined to be the number of integers between $1$ and $p^\nu$ relatively prime to $p^\nu$.
\end{lemma}

\Cref{On} gives the numerator and denominator of the Bernoulli number if we consider it as an irreducible fraction.
In particular, the largest prime factor of the denominator of the $n$-th Bernoulli number $B_n$ is less than or equal to $n+1$, which will play an important role in the following proof.

\begin{proof}[Proof of \Cref{NOsign1Ahat0}]
(1) If $k$ is odd, then by the equations in \eqref{eq-signM}, \eqref{eq-Ahatgenus}, 
\begin{equation}\label{eq-koddsignAhat}
\sign(M)=\frac{s^{}_k}{a^{}_k}\hat{A}(M)=-2^{2k+1}(2^{2k-1}-1)\hat{A}(M).
\end{equation}
It's obviously that $\sign(M)=1$ and $\hat{A}(M)=0$  cannot hold at the same time. 
	
(2) If $k$ is even, suppose that
\begin{equation}\label{eq-systemofeqs-signAhat}
\left\{
\begin{aligned}
\sign(M) & =s^{}_{k}p^{}_{k}[M]+s^{}_{\frac{k}{2},\frac{k}{2}}p_{\frac{k}{2}}^2[M]=1,\\
 \hat{A}(M) & =a^{}_{k}p^{}_{k}[M]+a^{}_{\frac{k}{2},\frac{k}{2}}p_{\frac{k}{2}}^2[M]=0,
\end{aligned}
\right.
\end{equation}
where 
\[
\begin{aligned}
s^{}_k & =\dfrac{2^{2k}(2^{2k-1}-1)|B_{2k}|}{(2k)!}, & s^{}_{\frac{k}{2},\frac{k}{2}} & =\dfrac{s_{\frac{k}{2}}^2-s^{}_{k}}{2},  \\
a^{}_k & =-\dfrac{|B_{2k}|}{2(2k)!}, & a^{}_{\frac{k}{2},\frac{k}{2}}& =\dfrac{a_{\frac{k}{2}}^2-a^{}_{k}}{2}.
\end{aligned}
\]
It implies that 
\begin{align}
p_{\frac{k}{2}}^2[M] & =\left(\frac{k!}{2^{k-1}(2^{k}-1)|B_{k}|}\right)^2, \label{Pk2M} \\
p^{}_{k}[M] & =\frac{(2k)!|B_{k}|^2 +2(k!)^2|B_{2k}|}{2^{2k}(2^{k}-1)^2|B_{k}|^2 |B_{2k}|}.\label{P2kM} 
\end{align}
By \Cref{On}, let $N_n$ be the numerator of the Bernoulli number $B_n$, then  the Pontrjagin number 
\[
p_{\frac{k}{2}}^2[M]=\left(\frac{k!\mathop{\prod}\limits_{p-1 | k\atop p~\text{prime}}p}{2^{k-1}(2^{k}-1)|N_{k}|}\right)^2.
\]

Let $P(m)$ be the largest prime factor of the positive integer $m$.
The quantity $P(2^n-1)$ of the Mersenne number $2^n-1$ has been investigated by many authors.
For example, the best-known lower bound \cite{KF} 
\[
P(2^n-1)\geqslant 2n+1  \text{~for~}  n\geqslant 13
\]
is due to Schinzel \cite[Theorem 3]{Sch1962}.
Since the largest prime factor of the numerator of $p_{\frac{k}{2}}^2[M]$ is not greater  than $k+1$ and $P(2^{k}-1)\geqslant 2k+1$ for $k\geqslant 13$, so $p_{\frac{k}{2}}^2[M]$ can't be an integer for $k\geqslant 13$.

We can put $k=6$ into the equation \eqref{P2kM} of $p^{}_{k}[M]$ and the value of 
\[
p^{}_{6}[M]=\frac{158175}{691}
\]
 isn't an integer.
Similarly, we put $k=8,10,12$ into the equation \eqref{Pk2M} of $p_{\frac{k}{2}}^2[M]$ and the values of $p_{\frac{k}{2}}^2[M]$ are listed in the following table.
\begin{table}[h]
\centering
\caption{Calculations of $p_{\frac{k}{2}}^2[M]$ for $k=8, 10, 12$} 
\begin{tabular}{c|ccc}
  \toprule 
  $k$ & $8$ & $10$ & $12$  \smallskip \\
  \midrule 
 $p_{\frac{k}{2}}^2[M]$ & $\dfrac{396900}{289}$ & $\dfrac{8037225}{961}$  & $\dfrac{24312605625}{477481}$ \\
  \bottomrule
 \end{tabular}
\end{table}

These Pontrjagin numbers aren't  integers, so  $\sign(M)=1$ and $\hat{A}(M)=0$  cannot hold at the same time. 
\end{proof}

\begin{proposition}
Assume that  $M$ is an orientable closed  manifold of dimension $4k>48$, and all the Pontrjagin numbers of $M$ except maybe $p_{\frac{k}{2}}^2[M]$ ($k$ even) and $p^{}_{k}[M]$ vanish. If $\hat{A}(M)=0$, 
then $\sign(M)=0$ or $|\sign(M)|\geqslant (2k+1)^2$.
\end{proposition}

\begin{proof}
If $k$ is odd, the equation \eqref{eq-koddsignAhat} implies that $\hat{A}(M)=0$ is equivalent to $\sign(M)=0$.
If $k$ is even, the system of equations \eqref{eq-signM} and \eqref{eq-Ahatgenus} implies that 
\[
p_{\frac{k}{2}}^2[M]=\left(\frac{k!\mathop{\prod}\limits_{p-1 | k\atop p~\text{prime}}p}{2^{k-1}(2^{k}-1)|N_{k}|}\right)^2\cdot \sign(M).
\]
Since $p_{\frac{k}{2}}^2[M]$ is an integer and the largest prime factor of the positive integer of $2^k-1$ satisfies the inequality $P(2^{k}-1)\geqslant 2k+1$ for $k>12$, so the signature of $M$ is divisible by $(P(2^{k}-1))^2$. Hence $\sign(M)=0$ or $|\sign(M)|\geqslant (2k+1)^2$.
\end{proof}

\section{Circle actions on manifolds with three fixed points}\label{sec-S1mfd3fpts}

Let the circle act on a $2n$-dimensional oriented closed manifold $M$ with a discrete fixed points set. At each fixed point $x\in M^{S^1}$, the tangent space $T_xM$ splits as a sum of irreducible non-trivial complex $1$-dimensional $S^1$ representations
\[
T_xM=L^{w_{x,1}}\oplus\cdots \oplus L^{w_{x,n}},
\] 
where each $L^{a}$ denotes a $S^1$ representation on which $\lambda\in S^1$ acts by multiplication by $\lambda^a$. 
For each $L^{w_{x,i}}$, we choose an orientation of $L^{w_{x,i}}$ so that $w_{x,i}$ is positive, and call the positive integers $w_{x,1}, \dots, w_{x,n}$ the \emph{weights} (or \emph{exponents}) of the $S^1$-action at the (connected component of $M^{S^1}$ containing the) point $x$.
The tangent space $T_xM$ at $x$ has two orientations: one coming from the orientation on $M$ and the other coming from the orientation on the representation space $L^{w_{x,1}}\oplus\cdots \oplus L^{w_{x,n}}$.
Let $\epsilon(x)=+1$ if the two orientations agree and $\epsilon(x)=-1$ otherwise, and call it the \emph{sign} of $x$. Jang \cite{Jang2021}
defined the \emph{fixed point data} of $x$ by $\Sigma_x=\{\epsilon(x),w_{x,1},\dots, w_{x,n}\}$, 
and the \emph{fixed point data} of $M$ to be a collection of fixed point datum of the fixed points and denoted it by $\Sigma_M$. 

Note that $\sum_{i=1}^n\omega_{x,i}$ is constant on the same connected component of $M^{S^1}$.
A circle action is called \emph{$2$-balanced} if the parity of $\sum_{i=1}^n\omega_{x,i}$ does not depend on the connected component of $M^{S^1}$ (\cite[p.179]{HBJ92}).

\begin{lemma}\label{abc}\cite[Proposition 4.1 \& Lemma 6.1]{Jang2021}
Let the circle act on a smooth $4k$-dimensional oriented closed manifold $M$ with three fixed points $x_1$, $x_2$, $x_3$.
Then 
 
(1) the fixed point data of $M$ is
\begin{align*}
\Sigma_{x_1} & =\{\pm, a_1,\dots ,a_k, b_1,\dots , b_k\},\\
\Sigma_{x_2} & =\{\pm, a_1,\dots ,a_k, c_1,\dots , c_k\},\\
\Sigma_{x_3} & =\{\mp, b_1,\dots ,b_k, c_1,\dots , c_k\},
\end{align*}
for some positive integers $a_i$, $b_i$ and $c_i$, $1\leqslant i\leqslant k$.

(2) if $k\geqslant 3$, then $\sum_{i=1}^ka_i^2=\sum_{i=1}^kb_i^2=\sum_{i=1}^kc_i^2$.
\end{lemma}

Since squaring doesn't change the parity of a number,
hence, by \Cref{abc},  
\[
\sum_{i=1}^ka_i\equiv\sum_{i=1}^kb_i\equiv\sum_{i=1}^kc_i\pmod 2.
\]
We can get the following result.

\begin{corollary}\label{coro}
Let the circle act on a $4k$-dimensional oriented closed manifold $M$ with three fixed points and $k\geqslant 3$, then the circle action must be $2$-balanced.
\end{corollary}

It is a classical result of Atiyah and Hirzebruch \cite{AH} that if the circle acts non-trivially on a spin manifold $M$, then $\hat{A}(M)=0$.
The manifold with a $2$-balanced $S^1$-action is analogous to spin manifold.

\begin{lemma}\label{hat}\cite[Corollary 3.3]{HHH}
Let $M$ be a  smooth even-dimensional oriented closed manifold admitting a $2$-balanced $S^1$-action.
If the $S^1$-action is non-trivial, then 
\[
\hat{A}(M)=0.
\]
\end{lemma}

Make use of the equivariant characteristic numbers, Wiemeler showed that
\begin{lemma}\cite[Lemma 3.1]{Wiemeler23}\label{wie-Pk2P2k}
Let $M$ be a smooth orientable closed $S^1$-manifold of dimension $4k>0$ with exactly three fixed points.
Then all Pontrjagin numbers of $M$ except may be $p_{\frac{k}{2}}^2[M]$ ($k$ even) and $p^{}_k[M]$ vanish.
\end{lemma}
Note that, compared with \Cref{wie-Pk2P2k}, some non-zero Pontrjagin numbers of $S^1$-manifolds imply the lower bound of the number of the fixed points. Suppose $N^{4mn}$ is a smooth manifold, if there exists a partition $(\lambda_1,\dots, \lambda_u)$ of $m$ such that the corresponding Pontrjagin number $(p^{}_{\lambda_1}\cdots p^{}_{\lambda_u})^n[N]$ is non-zero, Li and Liu \cite[Theorem 1.4]{LiLiu2011} proved that any $S^1$-action on $N$ must have at least $n + 1$ fixed points (see also \cite[Example 3.4]{Li2012}).
So for the smooth orientable closed manifold $M$ of dimension $4k>0$, if $p_{\frac{k}{2}}^2[M]$ ($k$ even) is non-zero, then any $S^1$-action on $M$ must have at least $3$ fixed points.

Now, we show our main \Cref{thm-Noexis3fixedpts}.
\begin{proof}[Proof of \Cref{thm-Noexis3fixedpts}]
Suppose that there exists a non-trivial circle action on $M$ with exactly three fixed points. There are many methods (\cite{Li2015, Jang2021, Wiemeler23}) to prove that the signature of $M$ is equal to $\pm 1$. We choose a suitable orientation such that $\sign(M)=1$. 
The action has to be 2-balanced by \Cref{coro} and $\hat{A}(M)=0$ by \Cref{hat}.
By \Cref{wie-Pk2P2k}, all Pontrjagin numbers of $M$ except may be $p_{\frac{k}{2}}^2[M]$ ($k$ even) and $p^{}_{k}[M]$ vanish.  
But it is contrary to \Cref{NOsign1Ahat0}, since $\sign(M)=1$ and $\hat{A}(M)=0$ cannot hold at the same time if the manifold has at most two non-zero Pontrjagin numbers $p_{\frac{k}{2}}^2[M]$ ($k$ even) and $p^{}_{k}[M]$.
\end{proof}

\section{Circle actions on rational projective planes}\label{sec-RPP}

Let $\mathcal{A}=\oplus_{i=0}^n A^i$ be a $1$-connected ($A^1=0$) graded commutative algebra over $\mathbb{Q}$ satisfying the Poincare duality.
It is a natural question to ask if there exists a simply-connected closed $n$-manifold $M$ such that $H^*(M;\mathbb{Q})=\mathcal{A}$.
The case of $\mathcal{A}=\mathbb{Q}[x]/\langle x^3\rangle$, $|x|=2k$, was studied in \cite{Su14, FS, KenSu19}.
A closed smooth $4k$-manifold with such rational cohomology ring is called a \emph{rational projective plane}.  In the dimension $4$, $8$, and $16$ (the dimensions of $\cp^2$, $\mathbb{H}P^2$ and $\mathbb{O}P^2$ respectively), rational projective planes evidently exist. After dimension $4$, $8$, and $16$, the next smallest dimension where a rational projective plane exists is $32$ (\cite[Theorem 1.1]{Su14}).  Fowler and Su \cite[Theorem A]{FS} showed that for $n\geqslant 8$, a rational projective plane can only exist in dimension $n=8k$ with $k=2^a+2^b$ for some non-negative integers $a$, $b$.

The proof of the \Cref{thm-rppnoS1} is dependent on the following two results.
\begin{lemma}\label{KlauKrec}\cite[Lemma 2.4]{KlauKrec}
Let $X$ be a simply connected space with $\widetilde{H}_i(X;\mathbb{Q})=0$ for $i<r$.
Then there is a $(r-1)$-connected space $Y$ and a map $f: Y\rightarrow X$ inducing isomorphisms in all rational homotopy and homology groups.
\end{lemma}

\begin{lemma}\label{HH}\cite{HH, HH2012}
Let $M$ be a smooth $2n$-dimensional, closed, oriented and connected manifold with finite second and fourth homotopy groups and endowed with a non-trivial smooth $S^1$-action. Then 
\[
\hat{A} (M)=0.
\]
\end{lemma}

\begin{proof}[Proof of \Cref{thm-rppnoS1}]
Let $M$ be a $8k$-dimensional rational projective plane with $k>2$, the only possible non-zero Pontrjagin numbers are $p^{}_{2k}[M]$ and $p_k^2[M]$.
By the definition of the signature, the signature of any rational projective plane $M$ is $\pm 1$. We sign it as 
\[
\sign(M)=1.
\]
By Lemma \ref{KlauKrec}, $\pi _2(M)$ and $\pi _4(M)$ are both finite groups.
If $M$ admits a non-trivial smooth $S^1$-action,
 by Lemma \ref{HH} we have 
\[
\hat{A} (M)=0. 
\]  
According to \Cref{NOsign1Ahat0}, the conclusion is immediately followed.
\end{proof}

As to the highly connected manifolds, we refer it as the $(2k-1)$-connected $4k$-manifold, we arrive at the following result.
\begin{proposition}
Let $M$ be an orientable closed  $(2k-1)$-connected manifold with dimension $4k$, $k>2$ odd. If $M$ admits a non-trivial $S^1$-action, then 
 $M$ is rationally zero-bordant.
\end{proposition}
\begin{proof}
When $k$ is odd, the $(2k-1)$-connected $4k$-manifold $M$ has at most one non-zero Pontrjagin number $p^{}_k[M]$, and $M$ is spin when $k>2$. If $M$ admits a non-trivial $S^1$-action, then by the classical result of Atiyah and Hirzebruch, $\hat{A}(M)=0$. By the equation \eqref{eq-Ahatgenus}, it implies $p^{}_k[M]=0$, i.e. all Pontrjagin numbers of $M$ are zero. Due to \cite[Corollary 18.10]{MilnorStash}, then some positive multiple $M^{4k}+\cdots+M^{4k}$ is an oriented boundary.
\end{proof}

\end{document}